\documentclass{amsart}

\usepackage{amsmath,amssymb,amsthm,amscd}

\usepackage{color}
\def\red#1{{\textcolor{red}{#1}}} 
\def\blue#1{{\textcolor{blue}{#1}}} 
\definecolor{magenta}{rgb}{1,0,1}
\def\magenta#1{{\textcolor{magenta}{#1}}} 

\usepackage{pinlabel}

\newtheorem{prop}{Proposition}[section]
\newtheorem{thm}[prop]{Theorem}
\newtheorem{lem}[prop]{Lemma}

\theoremstyle{definition}

\newtheorem{rem}[prop]{Remark}

\newtheorem*{ack}{Acknowledgements}

\def\co{\colon\thinspace}

\newcommand{\bfn}{\mathbf{n}}
\newcommand{\new}{\mathrm{new}}

\newcommand{\RP}{\mathbb{R}\mathrm{P}}

\newcommand{\R}{\mathbb{R}}

\newcommand{\bfu}{\mathbf{u}}

\newcommand{\bfv}{\mathbf{v}}

\newcommand{\bfw}{\mathbf{w}}

\newcommand{\Z}{\mathbb{Z}}

\DeclareMathOperator{\Hom}{Hom}
\DeclareMathOperator{\Int}{Int}
\DeclareMathOperator{\sign}{sign}

\begin{document}

\author[S.~Durst]{Sebastian Durst}
\author[H.~Geiges]{Hansj\"org Geiges}
\address{Mathematisches Institut, Universit\"at zu K\"oln,
Weyertal 86--90, 50931 K\"oln, Germany}
\email{sdurst@math.uni-koeln.de}
\email{geiges@math.uni-koeln.de}

\author[M.~Kegel]{Marc Kegel}
\address{Institut f\"ur Mathematik, Humboldt-Universit\"at zu Berlin,
Unter den Linden 6, 10099 Berlin, Germany}
\email{kegemarc@math.hu-berlin.de}

\title{Handle homology of manifolds}

\date{}

\begin{abstract}
We give an entirely geometric proof, without recourse to
cellular homology, of the fact that $\partial^2=0$ in
the chain complex defined by a handle decomposition
of a given manifold. Topological invariance of the
resulting `handle homology' is a consequence of Cerf theory.
\end{abstract}



\maketitle


\section{Introduction}
It is well known that any handle decomposition of a given
smooth manifold gives rise to a cell complex of the same homotopy
type. This allows one to compute the singular homology of the manifold
as the cellular homology of that associated complex.
In this situation, the boundary operator in the cellular chain complex
can be interpreted as a boundary operator on handles,
defined in terms of intersection numbers between
the attaching spheres of $k$-handles and the belt spheres of
$(k-1)$-handles.

In the present note we take this geometric interpretation of the
boundary operator $\partial$ as our starting point, and we prove
$\partial^2=0$ by geometric means, rather than by relating $\partial$
to the boundary operator on the cellular chain complex.
As a benefit, the resulting `handle homology' no longer relies,
neither explicitly nor implicitly, on singular homology theory.
The topological invariance of the `handle homology' thus constructed
is proved by appealing to Cerf theory~\cite{cerf70},
according to which any two handle decompositions of a given manifold
are related by some simple moves, cf.~\cite[Theorem~4.2.12]{gost99}:
handle slides and the creation or annihilation
of cancelling handle pairs.

The definition of the homology of a manifold in terms of a
handle decomposition gives a very simple proof of Poincar\'e duality;
our reasoning puts this proof on a purely geometric
footing. Apart from the geometric proof of $\partial^2=0$,
which we have not found in the literature, this note
is largely expository, expanding on some aspects
of~\cite[Section~4.2]{gost99}. It is instructive to
compare our arguments for dealing with
sign issues in the proof of $\partial^2=0$
with the discussion of orientations in
Morse homology~\cite{schw93,webe06}.
\section{Handle decompositions}
We assume that the reader is familiar with the basics of
handle decompositions of manifolds at the level of
\cite[Sections 4.1]{gost99}; see also
\cite{geig17} for an elementary introduction.
Here we only recall the parts of this theory necessary to set up
notation.
\subsection{Handles}
An $n$-dimensional $k$-handle is a copy of $h_k:= D^k\times D^{n-k}$,
attached along its \emph{lower boundary} $\partial_-h_k:=\partial D^k\times
D^{n-k}$ to the boundary of a smooth $n$-dimensional manifold
$X$ by an embedding
\[ \varphi\co\partial_-h_k\longrightarrow\partial X,\]
see Figure~\ref{figure:k-handle}.
The number $k\in \{0,\ldots, n\}$ is called the \emph{index} of the handle.

\begin{figure}[h]
\labellist
\small\hair 2pt
\pinlabel $h_k$ at 193 201
\pinlabel ${\partial X}$ at 634 204
\pinlabel $X$ at 550 38
\pinlabel ${\partial_-h_k=\partial D^k\times D^{n-k}}$ [t] at 360 30
\pinlabel ${\partial_+h_k=D^k\times\partial D^{n-k}}$ [l] at 553 329
\pinlabel ${\{0\}\times\partial D^{n-k}}$ [r] at 211 325
\pinlabel ${D^k\times\{0\}}$ [l] at 600 285
\endlabellist
\centering
\includegraphics[scale=.4]{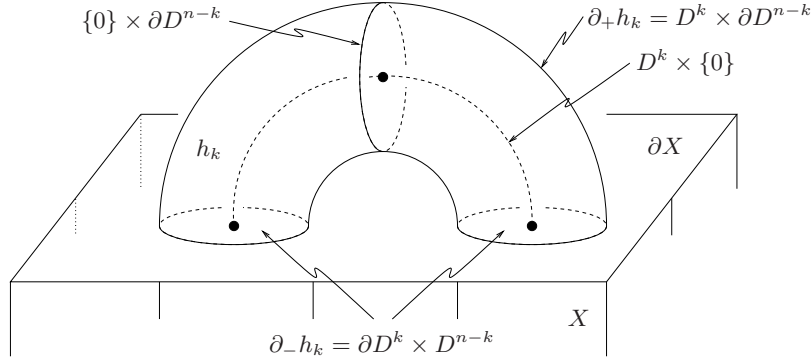}
  \caption{Attaching a $k$-handle.}
  \label{figure:k-handle}
\end{figure}

After smoothing the \emph{corner}
$\varphi(\partial D^k\times\partial D^{n-k})$,
the resulting space $X\cup_{\varphi}h_k$ is a smooth manifold. Its boundary
is given by removing $\varphi(\partial_-h_k)$ from $\partial X$
and replacing it with the \emph{upper boundary} $\partial_+h_k:=
D^k\times\partial D^{n-k}$.

We write
\[ A_k:=\partial D^k\times \{0\}\equiv\varphi(\partial D^k\times\{0\})
\cong S^{k-1} \]
for the \emph{attaching sphere} of the $k$-handle $h_k$, and
\[ B_k:=\{0\}\times\partial D^{n-k}\cong S^{n-k-1} \]
for its \emph{belt sphere}.

We shall usually identify
$A_k$ and $\partial_-h_k$ with their respective images in $\partial X$
under the embedding~$\varphi$.
\subsection{Nice handle decompositions}
\label{subsection:nice}
Let $M$ be a smooth compact $n$-manifold with boundary
$\partial M=\partial_-M\sqcup\partial_+M$, where either collection
$\partial_{\pm}M$ of boundary components may be empty.
If $M$ is oriented, the boundaries are oriented such that
$\partial M=\overline{\partial_-M}\sqcup\partial_+M$.
A \emph{handle decomposition} of $M$ relative to $\partial_-M$
is an identification of $M$ with a manifold obtained by
successively attaching handles to $[0,1]\times\partial_-M$
along $\{1\}\times\partial_-M$, where $\{0\}\times\partial_-M$
is identified with $\partial_-M$. Attaching a $0$-handle amounts to the
disjoint union with an $n$-disc $D^n=\{0\}\times D^n$.
Morse theory implies that such a handle decomposition
always exists.

If we attach a $k$-handle $h_k$ to an $n$-manifold $X$ with boundary,
followed by an $\ell$-handle
$h_{\ell}$ with $\ell\leq k$, the attaching sphere
$A_{\ell}$ in $\partial (X\cup h_k)$ can be made disjoint
from the belt sphere $B_k$ by an isotopy, since
\[ \dim A_{\ell}+\dim B_k=(\ell-1)+(n-k-1)<n-1=\dim\partial(X\cup h_k).\]
This allows one to push $\partial_-h_{\ell}$ away from
$\partial_+h_k$ by an isotopy that flows radially outward
(in the $D^k$-factor) on $\partial_+h_k$.
It follows that handles may always be attached in the order of
increasing index, and this will be assumed from now on.
We write $M_k$ for the manifold obtained from
$[0,1]\times\partial_-M$ by attaching handles up to and including index~$k$.
We also assume without loss of generality that the lower boundaries of
the $k$-handles are (disjointly) embedded in $\partial M_{k-1}\setminus
\partial_-M$, rather than some lower boundaries intersecting the
upper boundaries of other $k$-handles.

Now consider the attaching of a $k$-handle $h_k$, followed
by a $(k+1)$-handle $h_{k+1}$. Since
\[ \dim A_{k+1}+\dim B_k=k+(n-k-1)=n-1,\]
we may assume after an isotopy that $A_{k+1}$ and $B_k$
intersect each other transversely in finitely many points.
The embedding of the lower boundary $\partial_-h_{k+1}=\partial D^{k+1}
\times D^{n-k-1}$ defines a tubular neighbourhood of~$A_{k+1}$, and we can
take this to intersect $B_k$ in finitely many copies of
$\{*\}\times D^{n-k-1}$, with $*\in\partial D^{k+1}$.

By flowing out radially from $B_k$ on the upper boundary $\partial_+h_k$
we may further assume that $A_{k+1}$ intersects $\partial_+h_k$
in finitely many copies of $D^k\times\{*\}$, with $*\in\partial D^{n-k}$.

Finally, we consider the intersection of the attaching sphere
$A_{k+1}$ with the belt sphere $B_{k-1}$ of any $(k-1)$-handle
in $\partial M_k$. The dimensions of $A_{k+1}$ and $B_{k-1}$
add up to~$n$, so after making the intersection transverse,
it will be a $1$-dimensional manifold with boundary; the boundary points
lie in the corners of the $k$-handles.

These assumptions on the handle decomposition
being sufficiently `nice', illustrated in Figure~\ref{figure:A-sphere},
will be taken for granted from now on. The superscripts
$\mu$ and $\nu$ are used to label the $k$- and $(k-1)$-handles,
respectively.  Beware that, due to lack of dimensions, this figure is a
little misleading. The intersection $A_{k+1}\cap B_{k-1}^{\nu}$
is indeed $1$-dimensional, but the intersection $A_{k+1}\cap
\partial_+h_k^{\mu}$ is $k$-dimensional. Also, because of $k=1$
in the figure, the belt sphere $B_{k-1}^{\nu}$ coincides with
a component of $\partial M_{k-1}$,

\begin{figure}[h]
\labellist
\small\hair 2pt
\pinlabel $h_k^{\mu}$ [br] at 253 315
\pinlabel $A_{k+1}\cap\partial_+h_k^{\mu}$ [bl] at 536 315
\pinlabel $M_{k-1}$ at 550 38
\pinlabel $A_{k+1}\cap{B_{k-1}^{\nu}}$ [b] at 82 284
\pinlabel $b$ [tl] at 348 311
\pinlabel $a$ [r] at 212 127
\pinlabel $c$ [t] at 171 109
\endlabellist
\centering
\includegraphics[scale=.46]{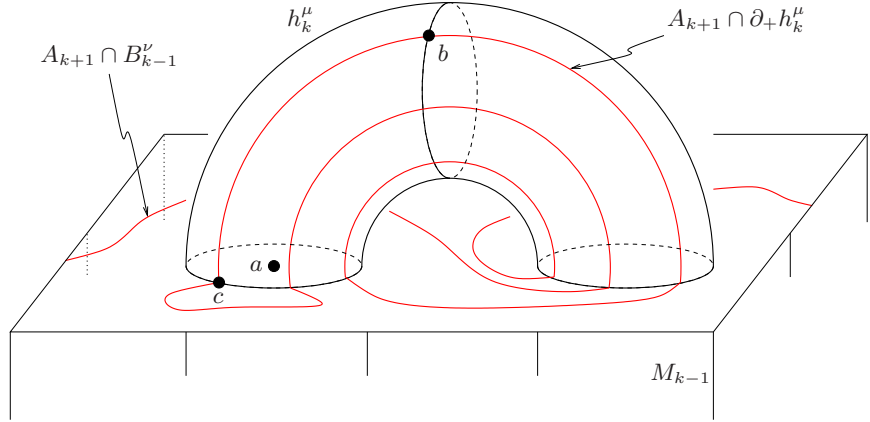}
  \caption{An attaching sphere $A_{k+1}$.}
  \label{figure:A-sphere}
\end{figure}

Let $c$ be a boundary point of the $1$-dimensional manifold
$A_{k+1}\cap B_{k-1}^{\nu}$. This point lies in the corner
$\partial D^k\times\partial D^{n-k}$ of a $k$-handle
$h_k^{\mu}$, and we write it as $c=(c_k,c_{n-k})$ with respect to
this product structure. Since our handle decomposition is nice,
we have
\[ \{c_k\}\times D^{n-k}\subset B_{k-1}^{\nu}\cap\partial_-h_k^{\mu}\]
and
\[ D^k\times\{c_{n-k}\}\subset A_{k+1}\cap\partial_+h_k^{\mu}.\]
This determines a pair of points $(a,b)$ with
\[ a=(c_k,0)\in B_{k-1}^{\nu}\cap A_k^{\mu}\]
and
\[ b=(0,c_{n-k})\in A_{k+1}\cap B_k^{\mu}.\] 
Conversely, any pair of points
\[ (a,b)\in (B_{k-1}^{\nu}\cap A_k^{\mu})\times (A_{k+1}\cap B_k^{\mu})\]
determines a unique point $c$ in the boundary
of $A_{k+1}\cap B_{k-1}^{\nu}$. So there is a one-to-one
correspondence between pairs $(a,b)$ and points $c\in\partial(A_{k+1}\cap
B_{k-1}^{\nu})$.

In Morse homology, each pair $(a,b)$ (or the respective point~$c$)
corresponds to a broken trajectory of the negative gradient flow,
connecting two critical points of index difference~$2$, broken
at a critical point of intermediate index.
\subsection{Orientations}
\label{subsection:orientations}
(1) If the manifold $M$ is oriented, we orient the $k$-handles
in the relative handle decomposition of $(M,\partial_-M)$
as follows. Choose any identification of the core disc of
$h_k$ with $D^k\times\{0\}$, equipped with
its standard orientation as unit disc in~$\R^k$
(see Figure~\ref{figure:k-handle}). Then identify the belt disc
of $h_k$ with $\{0\}\times D^{n-k}$ (and hence $h_k$ with
$D^k\times D^{n-k}$) in such a way that the
orientation of $h_k$ induced by the orientation of $M$ coincides
with the standard orientation of $D^k\times D^{n-k}$ as
subset of~$\R^n$.

The attaching sphere $A_k$ is oriented as the boundary of the
core disc $D^k\times\{0\}$ according to the dictum `outward normal first';
the belt sphere $B_k$, as the boundary of the belt disc $\{0\}\times D^{n-k}$.

The successive handlebodies $M_k$ are oriented as equidimensional
submanifolds of~$M$; the hypersurface $\partial M_k$ is oriented as
the boundary of~$M_k$.

Notice that for the $n$-handles the core disc coincides with
the full handle. In particular, here the orientation of
the core disc is determined by the ambient orientation,
and $A_n$ carries the opposite orientation of $\partial M_{n-1}$.

(2) If $M$ is non-orientable, we can still choose an orientation for
each core disc and give the attaching sphere the induced orientation.
This orientation of the $D^k$-factor in $h_k=D^k\times D^{n-k}$
defines a \emph{coorientation} on the belt sphere
$B_k=\{0\}\times\partial D^{n-k}$ as submanifold of
the upper boundary $\partial_+h_k$, i.e.\ an orientation of
the normal bundle of $B_k\subset\partial_+h_k$.
This, as we shall see, is sufficient
for defining the relevant intersection numbers over the integers.
\section{Handle homology}
We now use the handle decomposition of $M$ relative to
$\partial_-M$ to define relative homology groups $H_*(M,\partial_-M)$.
We first formulate everything under the assumption that $M$ is
oriented, where we can work with integral coefficients.
The arguments in Sections \ref{subsection:delta}
and~\ref{subsection:delsquared} likewise apply to non-orientable
manifolds if one works over~$\Z/2$. The necessary modifications to
define integral handle homology for non-orientable manifolds
will be explained in Section~\ref{subsection:integral}.
\subsection{Definition of the handle chain complex}
\label{subsection:delta}
Let $C_k(M,\partial_-M)$, be the free abelian group generated by the oriented
$k$-handles in a handle decomposition of~$M$ relative to $\partial_-M$. Of
course, $C_k(M,\partial_-M)$ depends on the choice of handle
decomposition, but we suppress this from the notation.

The boundary operator is then defined by
\[ \begin{array}{rccl}
\partial_k\co & C_k (M,\partial_-M) & \longrightarrow
    & C_{k-1}(M,\partial_-M)\\
              & h_k                 & \longmapsto    
    & (-1)^{k-1}\sum_{\nu} (A_k\bullet B_{k-1}^{\nu}) h_{k-1}^{\nu}.
\end{array}\]
Here $A_k\bullet B_{k-1}^{\nu}$ is the intersection number
of $A_k$ and $B_{k-1}^{\nu}$ in the oriented manifold~$\partial M_{k-1}$.

The sign $(-1)^{k-1}$ in this definition is chosen for consistency with
the boundary operator in cellular homology. For a point
$*\in\partial D^{n-k+1}$, the intersection number
of $D^{k-1}\times\{*\}\equiv D^{k-1}$ with the belt sphere
$\{0\}\times\partial D^{n-k+1}\equiv\partial D^{n-k+1}$ in
$\partial h_{k-1}$ equals
\[ (D^{k-1}\times\{*\})\bullet \partial D^{n-k+1}=(-1)^{k-1},\]
since the orientation of $D^{k-1}$ followed
by the orientation of $\partial D^{n-k+1}$ equals $(-1)^{k-1}$ times
the boundary orientation of $\partial_+h_{k-1}$.

To see this formally, we introduce the notation $[...]$
for orientations determined by submanifolds or frame fields
whose dimensions add up to the dimension of the ambient manifold.
For instance, in $h_{k-1}$ we have $[D^{k-1},D^{n-k+1}]=1$.
Writing $\bfn$ for the outer normal to
$\partial_+h_{k-1}=D^{k-1}\times\partial D^{n-k+1}$,
we have $[\bfn,\partial_+h_{k-1}]=1$.

The intersection number in question is then computed as
\begin{eqnarray*}
D^{k-1}\bullet\partial D^{n-k+1} & = & [\bfn,D^{k-1},\partial D^{n-k+1}]\\
 & = & (-1)^{k-1}[D^{k-1},\bfn,\partial D^{n-k+1}]\\
 & = & (-1)^{k-1}[D^{k-1},D^{n-k+1}]\;=\;(-1)^{k-1}.
\end{eqnarray*}

\begin{rem}
In \cite[p.~111]{gost99}, the boundary operator $\partial_k$
is defined without the factor $(-1)^k$. This does not change the
homology of the chain complex.
\end{rem}
\subsection{Proof of $\partial^2=0$}
\label{subsection:delsquared}
Applying the boundary operator twice, we find
\[ \partial_k(\partial_{k+1}h_{k+1})=-\sum_{\mu,\nu}
(A_{k+1}\bullet B_k^{\mu})(A_k^{\mu}\bullet B_{k-1}^{\nu})h_{k-1}^{\nu}.\]
We therefore need to show that
\[ \sum_{\mu}(A_{k+1}\bullet B_k^{\mu})(A_k^{\mu}\bullet B_{k-1}^{\nu})=0.\]
In the notation of Section~\ref{subsection:nice}, this amounts to
\[ \sum\sign(a)\sign(b)=0,\]
where the sum is over all intersection points
$a\in A_k^{\mu}\cap B_{k-1}^{\nu}$ and $b\in A_{k+1}\cap B_k^{\mu}$,
and $\sign\in\{\pm1\}$ denotes the sign
induced by the orientations of the relevant sub\-manifolds intersecting
transversely in the given point.

Now, each pair of points $(a,b)$ corresponds to
a boundary point $c$ of the compact $1$-dimensional
manifold $A_{k+1}\cap B_{k-1}^{\nu}$. This already proves $\partial^2=0$
over~$\Z_2$. Over the integers, our aim is to show that
$A_{k+1}\cap B_{k-1}^{\nu}$ can be oriented in such a way that
any boundary point $c$, with $\sign(c)$ denoting the boundary
orientation, satisfies
\[ \sign(c)=\sign(a)\sign(b).\]
Then the sum $\sum\sign(a)\sign(b)$ amounts to a sum
over the signed boundary points~$c$, and hence equals zero.

Consider a component of the $1$-manifold $C:=A_{k+1}\cap B_{k-1}$
(we suppress the superscript $\nu$ from now on) with a boundary
point~$c$. This point lies on the corner $S^{k-1}\times S^{n-k-1}$
of a $k$-handle~$h_k$. This corner is a separating hypersurface
in $\partial M_{k-1}$, with the lower boundary $\partial_-h_k$
to one side of it; see Figure~\ref{figure:corner}.


\begin{figure}[h]
\labellist
\small\hair 2pt
\pinlabel $S^{n-k-1}$ [br] at 499 443
\pinlabel $S^{k-1}$ [bl] at 247 338
\pinlabel $c$ [bl] at 441 262
\pinlabel $\mbox{\magenta{$A_{k+1}\cap{B_{k-1}}$}}$ [tl] at 238 153
\pinlabel $\bfu$ [tl] at 374 180
\pinlabel $\bfv$ [t] at 330 264 
\pinlabel $\bfw$ [r] at 426 348
\pinlabel $\mbox{\red{$A_{k+1}\cap\partial{M_{k-1}}$}}$ [tl] at 100 217
\pinlabel $\mbox{\blue{$\partial_-h_k\cap{B_{k-1}}$}}$ [r] at 691 415
\endlabellist
\centering
\includegraphics[scale=.4]{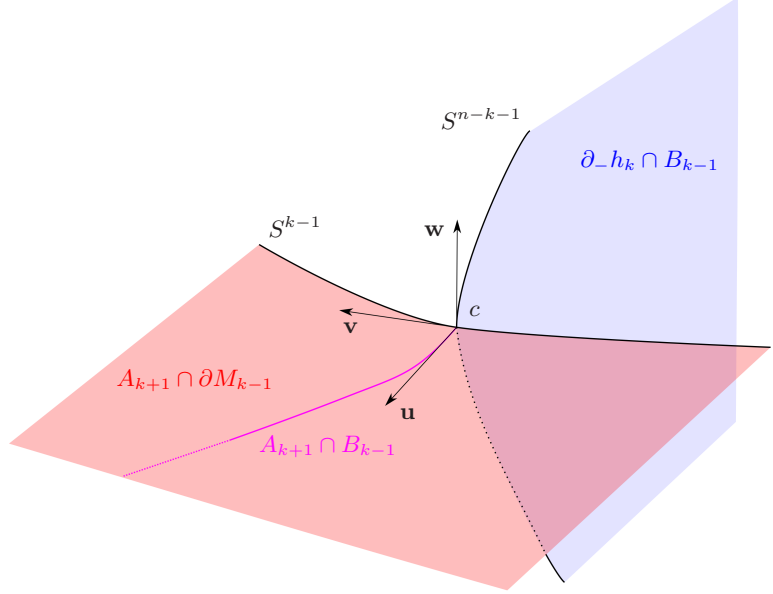}
  \caption{Orienting $A_{k+1}\cap B_{k-1}$.}
  \label{figure:corner}
\end{figure}

We now orient the $1$-manifold $C=A_{k+1}\cap B_{k-1}$ by a tangent
vector~$\bfu$ such that transverse frames $\bfv=(v_2,\ldots,v_k)$ in
$A_{k+1}$ and $\bfw=(w_2,\ldots,w_{n-k})$ in $B_{k-1}$ can be chosen
subject to the following orientation conventions:
\begin{itemize}
\item[(i)] $\bfv,\bfu$ is a positive frame for $A_{k+1}$ along~$C$;
\item[(ii)] $\bfu,\bfw$ is a negative frame for $B_{k-1}$ along~$C$;
\item[(iii)] $\bfv,\bfu,\bfw$ is a negative frame for
$\partial M_{k-1}$ along~$C$.
\end{itemize}
If $1<k<n-1$, in which case both $\bfv$ and $\bfw$ contain at least
one vector, these conventions indeed determine the orientation $[\bfu]$ of~$C$.
Simply choose a tangent vector $\bfu$ along $C$, and extend to frames
$\bfv,\bfu$ and $\bfu,\bfw$ subject to conditions (i) and~(ii).
If (iii) is satisfied by $\bfv,\bfu,\bfw$, this $[\bfu]$ is the orientation
we take for~$C$; if not, replace $\bfu$ by $-\bfu$, and one vector each
in $\bfv$ and $\bfw$ by its negative.

If one or both of $\bfv,\bfw$ are empty frames, i.e.\ if $k=1$ or $k=n-1$,
conditions (i) to (iii) are consistent: $A_n$ has the opposite
orientation of $\partial M_{n-2}$; the belt sphere $B_0$ is
a component of $\partial M_0$.

The point $c$ is then oriented as a boundary point of~$C$.
Figure~\ref{figure:corner} shows the case $\sign(c)=-1$.

\begin{lem}
\label{lem:sign}
$\sign(c)=\sign(a)\sign(b)$.
\end{lem}

\begin{proof}
For simplicity of notation, assume that the intersections $A_k\cap B_{k-1}$
and $A_{k+1}\cap B_k$ consist of single points
\[ A_k\cap B_{k-1}=\{a\},\;\;\; A_{k+1}\cap B_k=\{b\}.\]
Let $c\in\partial C$ be the corresponding boundary point of~$C$.

First we are going to relate the sign of $b$ to the orientation of the
$D^k$-factor in the $k$-handle~$h_k$.
Write $\bfn$ for the outer normal to $\partial M_{k-1}$. We may
assume that $h_k$ is attached vertically to
$M_{k-1}$, so that $\bfn$ may be thought of as a tangent vector to
$\partial_+h_k$ at~$c$; see Figure~\ref{figure:sign}, which shows the
situation for $\sign(c)=1$.

According to our conventions, $[\bfv,\bfu]$ is the positive orientation
of $A_{k+1}$ along~$C$. The transverse frame $\bfv$ can be chosen such
that at $c$ it is tangent to the $S^{k-1}$-factor of the corner
of~$h_k$, see Figure~\ref{figure:corner}. Then $[\bfv,\bfn]$
is likewise a frame of $A_{k+1}$ at~$c$, when we consider the
part of $A_{k+1}$ lying in~$\partial_+h_k$. The orientations
defined by these two frames are related by a factor $\sign(c)$.

\begin{figure}[h]
\labellist
\small\hair 2pt
\pinlabel $M_{k-1}$ at 111 37
\pinlabel $A_{k+1}\cap{B_{k-1}}$ [bl] at 43 128
\pinlabel $\bfu$ [t] at 247 121
\pinlabel $\bfv$ [tr] at 180 152
\pinlabel $\bfv$ [tr] at 270 152
\pinlabel $\bfn$ [l] at 289 172
\pinlabel $h_k$ [bl] at 365 113
\pinlabel $A_{k+1}$ [bl] at 264 238
\pinlabel $A_{k+1}$ [bl] at 120 96
\pinlabel $c$ [tr] at 289 125
\endlabellist
\centering
\includegraphics[scale=.6]{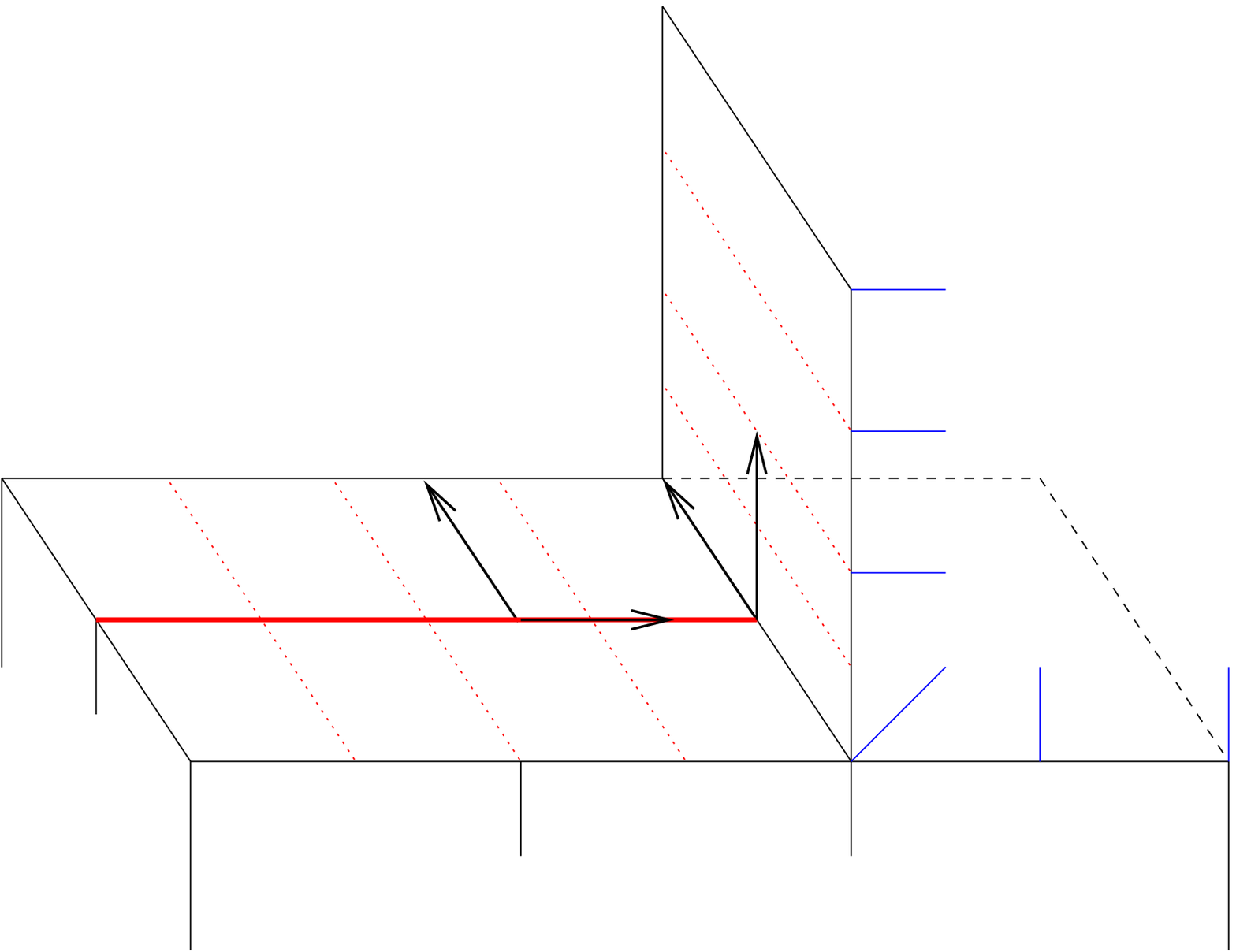}
  \caption{$\bfu$, $\bfn$ and $\sign(c)$.}
  \label{figure:sign}
\end{figure}

The part of $A_{k+1}$ lying in $\partial_+h_k$ is a $k$-disc
$D^k\times\{*\}$, with $*\in\partial D^{n-k}$, passing through the
point $b=A_{k+1}\cap B_k$ on the belt disc of~$h_k$.
By definition, we have $\sign(b)=A_{k+1}\bullet B_k$.
From the observation in Section~\ref{subsection:delta}
that $(D^k\times\{*\})\bullet B_k=(-1)^k$, we then deduce that
the positive orientation of $D^k\times\{*\}$ is given by
\[ (-1)^k\sign(b)\sign(c)[\bfv,\bfn]=-\sign(b)\sign(c)[\bfn,\bfv].\]

Next we relate the sign of $a$ to the orientation of the
$D^{n-k}$-factor in~$h_k$.
The part of $B_{k-1}$ lying in $\partial_-h_k$ is an $(n-k)$-disc
$\{*\}\times D^{n-k}$, with $*\in\partial D^k$,
passing through the point $a=A_k\cap B_{k-1}$.
By definition, we have
$\sign(a)=A_k\bullet B_{k-1}$, where the orientation of
the intersection point is determined with respect to the ambient
orientation of $\partial M_{k-1}$. Thus, regarded as an
intersection in $\partial_-h_k$, with respect to the
orientation as boundary of~$h_k$, the sign of the intersection
point is $-\sign(a)$.

Now, we have $\partial D_k\bullet(\{*\}\times D^{n-k})=1$
in $\partial_-h_k$. Recall that our convention (ii) says
that $B_{k-1}$ is oriented by $-[\bfu,\bfw]$, and $A_k$ is oriented
as~$\partial D^k$. It follows that the $D^{n-k}$-factor of $h_k$
is oriented by
\[ \sign(a)[\bfu,\bfw].\]

Assembling this information, we find that the orientation of $h_k$
is given by
\[ -\sign(a)\sign(b)\sign(c)[\bfn,\bfv,\bfu,\bfw].\]
On the other hand, the orientation of $M_{k-1}$, by~(iii), is
given by $-[\bfn,\bfv,\bfu,\bfw]$. We conclude
$\sign(a)\sign(b)\sign(c)=1$.
\end{proof}
\subsection{Integral homology of non-orientable manifolds}
\label{subsection:integral}
We now show that the orientation convention from
Section~\ref{subsection:orientations}~(2) suffices to define
integral homology over the integers, even if the manifold is
not orientable.

Let $A,B$ be submanifolds of a manifold $N$ with a transverse intersection
$A\cap B$. The manifold $N$ is not assumed to be orientable.
If $A$ is oriented and $B$ is cooriented (i.e.\ the normal bundle
of $B$ in $N$ is oriented), the submanifold $A\cap B$ inherits an
orientation, since the normal bundle of $A\cap B$ in $A$ can be
identified with the normal bundle of $B$ in $N$, restricted
to $A\cap B\subset B$. The intersection $A\cap B$
is oriented by this rule: the coorientation of $B$, followed by
the positive orientation of $A\cap B$,
defines the orientation of~$A$.

This means that we can still make sense of the intersection numbers
$A_k\bullet B_{k-1}$ of the submanifolds $A_k, B_{k-1}\subset
\partial M_{k-1}$. So we may define the boundary operator
$\partial_k$ as in Section~\ref{subsection:delta}. (With the
natural coorientation of the $B_{k-1}^{\nu}$, the factor $(-1)^{k-1}$
in the definition of $\partial_k$ can be removed; see Remark~\ref{rem:sign}.)

Likewise, the $1$-manifold $C=A_{k+1}\cap B_{k-1}$ inherits
an orientation~$[\bfu]$.
If $k=1$ (when $\dim A_{k+1}=1$), we simply take $\bfu$ as a vector
field defining the orientation of~$A_2$. If $k>1$ we choose $\bfu$
by the following rule. As before, $\bfv=(v_2,\ldots,v_k)$ denotes a
transverse frame along $C$ in~$A_{k+1}$, which we now interpret as a coframe
for $B_{k-1}$ in $\partial M_{k-1}$ along~$C$.
\begin{itemize}
\item[(i)] $\bfv$ is a positive coframe for $B_{k-1}$ along~$C$;
\item[(ii)] $\bfv,\bfu$ is a positive frame for $A_{k+1}$ along~$C$.
\end{itemize}

We can now prove Lemma~\ref{lem:sign} in this setting (up to an irrelevant
sign, see Remark~\ref{rem:sign}).
As before, we assume for notational simplicity that we are dealing
with a pair $(a,b)$ of single intersection points corresponding to
a boundary point $c$ of~$C$. The signs of these three points are given by
\begin{eqnarray*}
\sign(a)      & = & A_k\bullet B_{k-1},\\
\sign(b)      & = & A_{k+1}\bullet B_k,\\
{[\bfv,\bfn]} & = & \sign(c)[\bfv,\bfu].
\end{eqnarray*}
Here the first two equations hold by definition; for the third the
argument is as in the proof of Lemma~\ref{lem:sign}.

The coorientation of $B_k$ in $\partial M_k$ is given by
the orientation of $D^k\times\{*\}$ in the boundary of the $k$-handle~$h_k$.
Since the outer normal along the boundary of that disc is $-\bfn$
(see Figure~\ref{figure:sign}), this coorientation is $[-\bfn, A_k]$.
On the other hand, by (ii) and the definition of $\sign(b)$, this
coorientation is also given by $\sign(b)[\bfv,\bfu]$, hence
\[ [\bfv,\bfu]=\sign(b)[-\bfn, A_k].\]
Similarly, by (i) and the definition of $\sign(a)$, we have
\[ [A_k]=\sign(a)[\bfv].\]
We then compute
\begin{eqnarray*}
[\bfv,\bfu] & = & \sign(b)[-\bfn, A_k]\\
            & = & -\sign(a)\sign(b)[\bfn,\bfv]\\
            & = & (-1)^k\sign(a)\sign(b)[\bfv,\bfn]\\
            & = & (-1)^k\sign(a)\sign(b)\sign(c)[\bfv,\bfu],
\end{eqnarray*}
which proves that $\sign(c)=(-1)^k\sign(a)\sign(b)$. Hence
$\sum\sign(a)\sign(b)=0$ also in the non-orientable setting.

\begin{rem}
\label{rem:sign}
If the belt sphere $\{0\}\times \partial D^{n-k}$ in
$h_k=D^k\times D^{n-k}$ is cooriented by the orientation of~$D^k$,
as seems natural, then $D^k\bullet\partial D^{n-k}=1$, whereas in the
oriented setting we computed this intersection number as $(-1)^k$,
see Section~\ref{subsection:delta}. This accounts
for the extra sign in the above computation.
\end{rem}
\section{Topological invariance}
We define the `handle homology' $H_*(M,\partial_-M)$ as the
homology of the chain complex $C_*(M,\partial_-M)$.
By construction, it is isomorphic to the cellular homology
of the cell complex associated with the handle decomposition,
and hence a homotopy invariant.
It has already been observed in \cite[Section~4.2]{gost99} that,
alternatively, the following theorem of Cerf~\cite{cerf70} may be used to
show that this handle homology does not depend on the choice of relative
handle decomposition.

\begin{thm}[Cerf]
\label{thm:Cerf}
Any two relative handle decompositions of a compact manifold pair
$(M,\partial_-M)$, both ordered by increasing index,
are related by a finite sequence of handle slides and the
creation or annihilation of cancelling handle pairs.
\end{thm}

Cerf actually deals with homotopies of Morse functions.
Cancelling pairs of critical points in such a homotopy translate into
a cancelling handle pair; connecting trajectories between critical
points of equal index correspond to the intermediate stage
of a handle slide when the attaching sphere of a $k$-handle
passes through the belt sphere of another $k$-handle.

As sketched in \cite[p.~112]{gost99}, these moves do not affect
the handle homology. This means that $H_*(M,\partial_-M)$
is a diffeomorphism invariant. For completeness, we provide a few details.
\subsection{Handle slides}
Given two $k$-handles $h_k$ and $h_k'$ attached to the
boundary of an $n$-manifold~$X$, a handle slide of $h_k$
over $h_k'$ is defined by isotoping the attaching sphere
$A_k$ of $h_k$ over a $k$-disc $D^k\times\{*\}$ in the
upper boundary $\partial_+h_k'$ of~$h_k'$, as sketched
in Figure~\ref{figure:slide0}. Depending on the relative
orientations of $A_k$ and $\partial(D^{k}\times\{*\})$,
this amounts to replacing $A_k$ by the connected sum
$A_k\#\bigl(\pm\partial (D^k\times\{*\})\bigr)$, which in $\partial X$ is
isotopic to $A_k^{\new}:=A_k\#(\pm A_k')$. For the $k$-handle
$h_k^{\new}$ after the handle slide this implies
\[ \partial_k h_k^{\new}= \partial_k(h_k\pm h_k').\]

\begin{figure}[h]
\labellist
\small\hair 2pt
\endlabellist
\centering
\includegraphics[scale=.5]{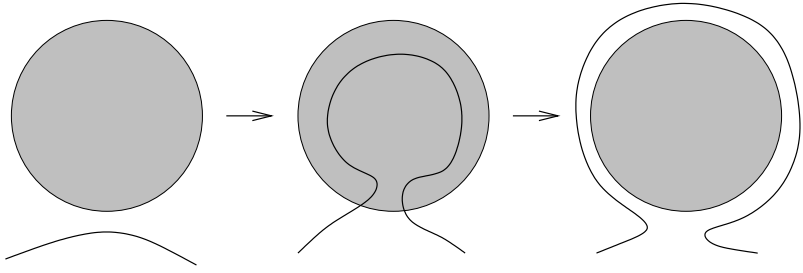}
  \caption{Sliding $A_k$ over $D^{k}\times\{*\}$.}
  \label{figure:slide0}
\end{figure}

In order to understand the effect of such a handle slide on the
handle homology, we need to take into account that the
$(k+1)$-handles that intersect the upper boundary $\partial_+h_k$ of
$h_k$ will be transformed by the handle slide.
A very simple example is shown in Figures \ref{figure:slide1}
and~\ref{figure:slide2}. With respect to the orientations of
the $1$-handles shown in the figures, we can interpret
$h_1^{\mathrm{new}}$ as the sum of $h_1$ and $h_1'$. Indeed, we have
$\partial_1h_1^{\new}=\partial_1(h_1+h_1')$.

\begin{figure}[h]
\labellist
\small\hair 2pt
\pinlabel $h_1'$ at 199 50
\pinlabel $h_1$ at 36 203
\pinlabel $\mbox{\red{$A_2$}}$ [bl] at 49 131
\pinlabel $\mbox{\blue{$A_2'$}}$ [br] at 372 98
\endlabellist
\centering
\includegraphics[scale=.6]{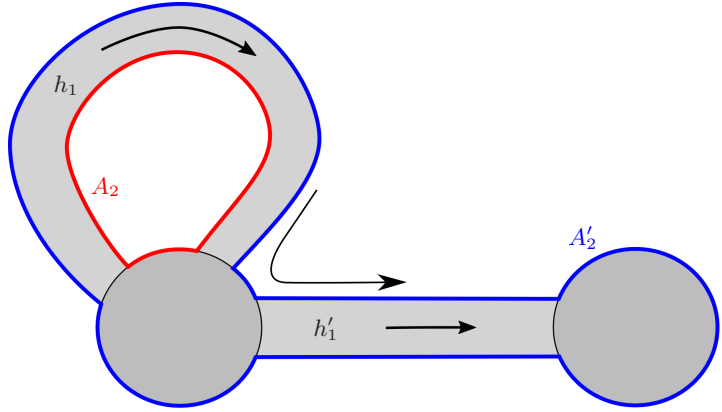}
  \caption{Before the handle slide.}
  \label{figure:slide1}
\end{figure}

The $2$-handle $h_2$ with attaching circle
$A_2$ shown in Figure~\ref{figure:slide1}
before the handle slide, with the standard orientation
of~$\R^2$, satisfies $\partial_2 h_2=-h_1$. After the handle slide
we have $\partial_2 h_2^{\new}=h_1'-h_1^{\new}$. Notice that we can also
write $\partial_2 h_2=h_1'-(h_1+h_1')$. In other words,
the boundary operator before the handle slide has exactly the same form
as the boundary operator after the handle slide, when in the former
we replace the basis element $h_1$ by $h_1+h_1'$. Similar
relations hold for the $2$-handle with attaching circle~$A_2'$.

\begin{figure}[h]
\labellist
\small\hair 2pt
\pinlabel $h_1'$ at 143 48
\pinlabel $h_1^{\new}$ at 126 188
\pinlabel $\mbox{\red{$A_2$}}$ [tl] at 100 156
\pinlabel $\mbox{\blue{$A_2'$}}$ [bl] at 372 88
\endlabellist
\centering
\includegraphics[scale=.6]{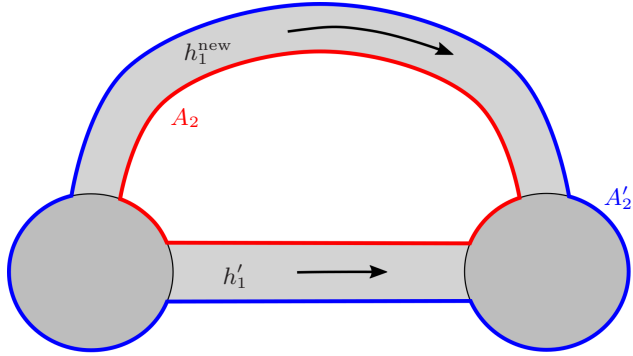}
  \caption{After the handle slide.}
  \label{figure:slide2}
\end{figure}

We now show that this is the general picture. Consider the attaching sphere
$A_{k+1}$ of a $(k+1)$-handle $h_{k+1}$. The part $A_{k+1}\cap
\partial_+h_k$ in the upper boundary of the sliding handle~$h_k$,
which is a collection of $m:=|A_{k+1}\cap B_k|$ copies
of~$D^k$, is simply moved along with~$h_k$. The part
$A_-:=A_{k+1}\setminus\Int(\partial_+h_k)$ of the attaching sphere
is a surface of genus zero with $m$ boundary circles.
The orientation of each of these circles, as boundary of~$A_-$,
is the \emph{opposite} of its orientation as the boundary of the
respective component $D^k$ of $A_{k+1}\cap\partial_+h_k$.
After the handle slide, we obtain the
boundary connected sum of $A_-$ with $m$ copies of $D^k\times\{*\}$
in $\partial_+h_k'$, where the orientation of $D^k\times\{*\}$
depends on the orientation of the corresponding boundary
component of~$A_-$. Our comment on the orientation
of $\partial A_-$ means that for the computation of the handle
homology, after the handle slide $h_k\rightsquigarrow h_k\pm h_k'$,
we may write
\[ A_-^{\new}=A_-\natural\bigl(\mp(-1)^k(A_{k+1}\bullet B_k)
(D^k\times\{*\})\bigr).\]
Notice the slight abuse of notation: we have $m=|A_{k+1}\cap B_k|$
disjoint copies of $D^k\times\{*\}\subset\partial_+h_k'$, but for
the computation of the boundary operator, only the signed count
$(-1)^k(A_{k+1}\bullet B_k)$ of these discs matters.

Since $(-1)^k(D^k\times\{*\})\bullet B_k'=1$, the boundary
$\partial_{k+1}h_{k+1}^{\new}$ is given by
\[ \partial_{k+1}h_{k+1}^{\new}=(-1)^k(A_{k+1}\bullet B_k)(h_k^{\new}\mp h_k')
+(-1)^k (A_{k+1}\bullet B_k')h_k'+\ldots\]
On the other hand, we have
\begin{eqnarray*}
\partial_{k+1}h_{k+1}
  & = & (-1)^k(A_{k+1}\bullet B_k)h_k+
        (-1)^k(A_{k+1}\bullet B_k')h_k'+\ldots\\
  & = & (-1)^k(A_{k+1}\bullet B_k)\bigl((h_k\pm h_k')\mp h_k'\bigr)+
        (-1)^k(A_{k+1}\bullet B_k')h_k'+\ldots
\end{eqnarray*}
Thus, as in our simple example, the new chain complex is given by
replacing the basis element $h_k$ with $h_k\pm h_k'$ in the old one.

One further comment is in order. In the case $k=n-1$, the
upper boundary $\partial_+h_k'=\partial_+h_{n-1}'$ is disconnected. If
$D^{n-1}\times\{*\}\subset\partial_+h_{n-1}'$ is a part of~$A_-$,
the slide of $h_{n-1}$ over $h_{n-1}'$, which formally would
lead to a boundary connected sum of $A_-$ with $D^{n-1}\times\{*\}$
with the reversed orientation, amounts to a subtraction of
$D^{n-1}\times\{*\}$ from~$A_-$. This is illustrated by
the handle with attaching circle $A_2'$ in Figures \ref{figure:slide1}
and~\ref{figure:slide2}.
\subsection{Cancelling handle pairs}
A cancelling pair of a $(k-1)$- and a $k$-handle, attached to the
boundary of an $n$-manifold $X$, can be described as follows,
see Figure~\ref{figure:cancelling}.

\begin{figure}[h]
\labellist
\small\hair 2pt
\pinlabel $D^k_0$ at 216 130
\pinlabel $D^{k-1}_-$ [t] at 139 22
\pinlabel $D^{k-1}_+$ [l] at 364 244
\pinlabel $D^1$ [l] at 216 237
\pinlabel $-1$ [t] at 216 216
\pinlabel $+1$ [b] at 216 253
\pinlabel ${{}\times D^{n-k}}$ [l] at 379 183
\pinlabel $X$ at 395 32
\endlabellist
\centering
\includegraphics[scale=.6]{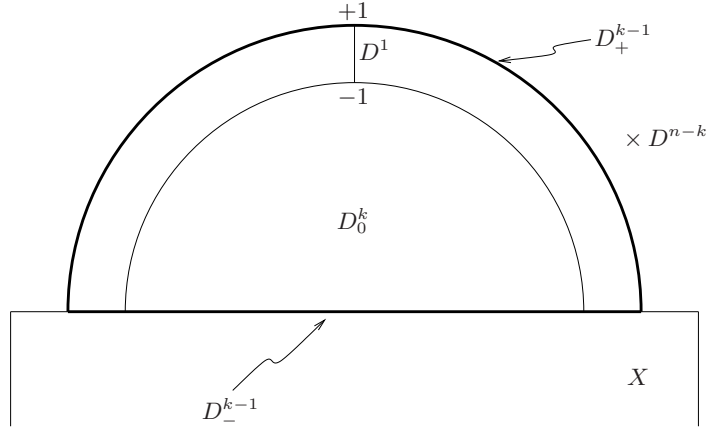}
  \caption{A cancelling handle pair.}
  \label{figure:cancelling}
\end{figure}

Write the boundary of the $k$-disc $D^k$ as the union of two discs:
\[ \partial D^k=D^{k-1}_+\cup D^{k-1}_-.\]
The boundary connected sum of $X$ with an $n$-ball,
which we think of as $D^k\times D^{n-k}$, along
$D^{k-1}_-\times D^{n-k}$
does not change~$X$ up to diffeomorphism, that is,
\[ X\cong X\cup_{D^{k-1}_-\times D^{n-k}} (D^k\times D^{n-k}).\]
We now want to interpret this boundary connected sum
as the successive attaching of a $(k-1)$- and a $k$-handle.
To this end, think of $D^k$ as the union of a smaller disc $D_0^k$,
whose boundary intersects $D^{k-1}_-$ in a slightly smaller disc,
and a rim $D^{k-1}_+\times D^1$, where $D^{k-1}_+\times\{1\}$
is identified with $D^{k-1}_+$, and $D^{k-1}_+\times\{-1\}$
with $\partial D_0^k\setminus D^{k-1}_-$:
\[ D^k=D^k_0\cup_{D^{k-1}_+\times\{-1\}} (D^{k-1}_+\times D^1).\]
Then 
\begin{eqnarray*}
X\cup_{D^{k-1}_-\times D^{n-k}} (D^k\times D^{n-k})
 & \cong & \bigl(X\cup_{\partial D^{k-1}_+\times D^1\times D^{n-k}}
             (D^{k-1}_+\times D^1\times D^{n-k})\bigr)\\
 &       & \mbox{}\cup_{\partial D^k_0\times D^{n-k}}(D^k_0\times D^{n-k}).
\end{eqnarray*}
Observe that the attaching sphere $\partial D^k_0\times\{0\}$
of the $k$-handle intersects the belt sphere $\{0\}\times
\partial(D^1\times D^{n-k})$ of the $(k-1)$-handle in a single point
\[ (0,-1,0)\in D^{k-1}_+\times D^1\times D^{n-k}.\]
Thus, if the boundary connected sum with the ball is done away from the
other handles, this interpretation as the introduction of
a cancelling handle pair does not change the handle homology.

Conversely, if a $(k-1)$-handle and a $k$-handle have been attached
such that $A_k$ intersects $B_{k-1}$ transversely in a single point, provided
the attaching is done `nicely' in the sense of Section~\ref{subsection:nice},
one can identify the handle attachment with this standard model
and hence remove this pair of handles.

We now want to describe the effect on the handle homology of
removing a cancelling handle pair $(h_k^{\mu_0},h_{k-1}^{\nu_0})$.
Here we need to take into account that the attaching sphere
$A_k$ may also pass over $(k-1)$-handles
other than~$h_{k-1}^{\nu_0}$.
Homotopically, this removal amounts to shrinking the core disc
$D_0^k\times\{0\}$ of $h_k^{\mu_0}$ to a point, so in the new chain
complex we simply set $h_k^{\mu_0}=0$. Since $A_k^{\mu_0}\bullet
B_{k-1}^{\nu_0}=\pm 1$, in the old chain complex we have
\[ \partial_k h_k^{\mu_0}=\pm h_{k-1}^{\nu_0} +(-1)^{k-1}\sum_{\nu\neq\nu_0}
(A_k^{\mu_0}\bullet B_{k-1}^{\nu})h_{k-1}^{\nu}.\]
The handle $h_{k-1}^{\nu_0}$ is removed from the chain complex by setting
the right-hand side of this equation equal to zero. Indeed, the
diffeomorphism $X\cong X\cup h_{k-1}^{\nu_0}\cup h_k^{\mu_0}$
amounts homotopically to isotoping the core disc
\[ D^{k-1}_+\equiv D^{k-1}_+\times\{(0,0)\}\subset
D^{k-1}_+\times D^1\times D^{n-k}\]
of $h_{k-1}^{\nu_0}$ rel boundary to $D^{k-1}_-$, where, as part
of the attaching sphere of~$h_k^{\mu_0}$, it represents
$\mp(-1)^{k-1}\sum_{\nu\neq\nu_0}(A_k^{\mu_0}\bullet
B_{k-1}^{\nu})h_{k-1}^{\nu}$ in the chain complex.

How does this removal of the cancelling handle pair affect the
other $k$-handles? Any $h_k^{\mu}$, $\mu\neq\mu_0$, may be assumed
to intersect $\partial_+h_{k-1}^{\nu_0}$ in discs of the form
\[ D^{k-1}_+\times\{(1,*)\}\subset
D^{k-1}_+\times D^1\times D^{n-k}.\]
Again, the diffeomorphism $X\cong X\cup h_{k-1}^{\nu_0}\cup h_k^{\mu_0}$
amounts homotopically to replacing each of these discs by~$D^{k-1}_-$.
In other words, in the boundary $\partial_k h_k^{\mu}$ for
$\mu\neq\mu_0$, we likewise have to replace $h_{k-1}^{\nu_0}$
by $\mp (-1)^{k-1}\sum_{\nu\neq\nu_0}
(A_k^{\mu_0}\bullet B_{k-1}^{\nu})h_{k-1}^{\nu}$ in the new chain complex.

We are now ready to show that this leaves the handle homology unchanged.
Write $C_*$ for the chain complex before, and $C_*'$ for the handle
complex after removing the cancelling handle pair. Let
\[ m_{\mu}:=(-1)^{k-1}(A_k^{\mu}\bullet B_{k-1}^{\nu_0})\]
be the coefficient of $h_{k-1}^{\nu_0}$ in the expansion
of~$\partial_k h_k^{\mu}$. As basis for $C_k$ we may use
\[ \bigl\{\text{$h_k^{\mu}\mp m_{\mu}h_k^{\mu_0}$ for $\mu\neq\mu_0$};\,
h_k^{\mu_0}\bigr\};\]
as basis for $C_{k-1}$ we choose
\[ \Bigl\{\text{$h_{k-1}^{\nu}$ for $\nu\neq\nu_0$};\,\pm h_{k-1}^{\nu_0}
+(-1)^{k-1}\sum_{\nu\neq\nu_0}
(A_k^{\mu_0}\bullet B_{k-1}^{\nu})h_{k-1}^{\nu}\Bigr\}.\]
The quotient groups $C_k'$ and $C_{k-1}'$ of $C_k$ and
$C_{k-1}$ under the new relations can be identified
with the subgroups given by removing the last vector in either basis.
Then the boundary operator $\partial_k\co C_k\rightarrow C_{k-1}$
splits as
\[ \partial_k=\partial_k'\oplus\mathrm{id}\co C_k'\oplus\Z\longrightarrow
C_{k-1}'\oplus\Z,\]
which shows that $H_*'=H_*$.
\subsection{Orientations}
In the case where $M$ is orientable, we made a
choice of orientations of $M$ and the core discs of the handles
to define the boundary operator over the integers,
but the resulting homology is independent of this choice.

Changing the orientation of $M$, while
keeping the orientations of the core discs, will change the
orientation of all belt spheres. This will leave the intersection
numbers $A_k\bullet B_{k-1}$ unchanged, because they are now
computed with respect to the opposite ambient orientation.

Changing the orientation of the core disc of a handle $h_k$,
while keeping the ambient orientation, will also change
the orientation of its belt sphere. So this change
amounts to replacing $h_k$ by $-h_k$ as a generator of~$C_k$.

When $M$ is not orientable, the argument why the choice of
(co-)orientations for the $A_k$ and $B_k$, respectively, does
not affect the resulting integral homology is analogous.
\subsection{Euler characteristic}
The Euler characteristic $\chi$ of a handle decomposition can be defined
as the alternating sum over the number of $k$-handles in
the decomposition. The topological invariance of $\chi$
is an immediate consequence of Theorem~\ref{thm:Cerf}.
This is a more direct argument for the invariance of $\chi$
than the usual algebraic reasoning~\cite[p.~146]{hatc02},
based on showing that $\chi$ equals the alternating sum over the ranks of
the homology groups.
\section{Poincar\'e duality}
\label{section:Poincare}
A relative handle decomposition of $(M,\partial_-M)$
can be read `upside down' as a relative handle decomposition
of $(M,\partial_+M)$. Any $k$-handle in the former becomes
an $(n-k)$-handle in the latter, with the roles of lower and upper
boundary, and that of attaching and belt sphere, reversed.

As has been observed before, see~\cite[p.~112]{gost99},
this yields a quick proof of Poincar\'e duality. We formulate it in the
oriented case for integral (co-)homology. For non-orientable manifolds it
remains true over~$\Z_2$.

\begin{thm}
Let $M$ be a compact, oriented $n$-manifold with boundary
$\partial M=\overline{\partial_-M}\sqcup\partial_+M$. Then
\[ H^k(M,\partial_-M)\cong H_{n-k}(M,\partial_+M).\]
\end{thm}

\begin{proof}
Reading the handle decomposition of $(M,\partial_-M)$ upside down
gives a dual chain complex $C_*'(M,\partial_+M)$ and the
following commutative diagram:
\[ \begin{CD}
C_k(M,\partial_-M)       @>{\partial_k}>>         C_{k-1}(M,\partial_-M)\\
@|                                                @|                    \\
C_{n-k}'(M,\partial_+M)  @<{\partial_{n-k+1}'} << C_{n-k+1}'(M,\partial_+M).
\end{CD}\]
The lower line computes $H_{n-*}(M,\partial_+M)$.

Since, in the dual handle decomposition, the roles
of attaching and belt spheres is reversed, the boundary operator
$\partial_{n-k+1}'$ is, up to sign, simply the transpose of~$\partial_k$.
Therefore, the lower line may be read as
\[ \Hom(C_k(M,\partial_-M),\Z)\stackrel{\partial_k^*}{\longleftarrow}
\Hom(C_{k-1}(M,\partial_-M),\Z).\]
So the lower line also computes $H^*(M,\partial_-M)$.
\end{proof}

\begin{rem}
\label{rem:RP2}
In the non-orientable situation, $\partial_{n-k+1}'$ is not, in general,
the transpose of~$\partial_k$ (even up to sign). For instance, in the
standard handle decomposition of the real projective plane
$\RP^2$ with a single $0$, $1$- and $2$-handle each,
$\partial_2,\partial_2'\co\Z\rightarrow\Z$ are multiplication by~$2$,
while $\partial_1$ and $\partial_1'$ are the zero homomorphism, so no new
information is gained. The reason is that the attaching sphere of the
$1$-handle $h_1$ counts as two points of opposite sign, while
the coorientations of the two points making up the belt sphere of $h_1$
both correspond to the same orientation of the attaching circle $A_2$
of the $2$-handle; see also the next section.
When turning this handle decomposition upside down,
this whole picture becomes reversed.
\end{rem}
\section{Orientability and top homology}
Let $M$ be a connected, closed $n$-dimensional manifold. We claim that
the top-dimensional integral handle homology of $M$ is given by
\[ H_n(M)\cong\begin{cases}
\Z & \text{if $M$ is orientable},\\
0  & \text{otherwise}.
\end{cases}\]

By handle cancellation we may assume that $M$ contains a single
$0$-handle~$h_0$. Each $1$-handle $h_1$ is attached by an embedding
$\partial_-h_1\rightarrow\partial h_0$. If each of these
embeddings is orientation preserving for the
boundary orientation of $\partial_-h_1\subset h_1$
(and a suitable orientation of~$h_1$), the manifold $M$
is orientable. If there is at least one embedding
$\partial_-h_1\rightarrow\partial h_0$ where the embeddings
of the two components of $\partial_-h_1$ have opposite
orientation behaviour, $M$ is not orientable.

We now compute the homology by turning this handle decomposition upside
down as in Section~\ref{section:Poincare}. The crucial
observation, as in Remark~\ref{rem:RP2}, is that $\partial_-h_1=
\{\pm 1\}\times D^{n-1}$, where the two $(n-1)$-discs inherit
\emph{opposite} orientations, while in the interpretation of $h_1$ as an
$(n-1)$-handle in the reverse handle decomposition,
\[ h_1=D^1\times D^{n-1}=h_{n-1}',\]
the belt sphere $B_{n-1}'=(\pm 1,0)$ consists of two points with
the \emph{same} coorientation. Thus, if $\partial_-h_1\rightarrow
\partial h_0$ is orientation preserving, we have $A_n'\bullet B_{n-1}'=0$;
if $\partial_-h_1\rightarrow \partial h_0$ maps the two components
with opposite orientation behaviour (with respect to the boundary orientation),
we have $A_n'\bullet B_{n-1}'=\pm 2$.

We conclude that $h_n'$ is a cycle generating the top-dimensional homology
in the orientable case, while $\partial_n'\neq 0$ in the
non-orientable case.
\begin{ack}
We thank the referee for very constructive comments. In particular,
the discussion of handle homology over the integers in the case
of non-orientable manifolds is based on a suggestion by the referee.
H.~G.\ is partially supported by the SFB/TRR 191
``Symplectic Structures in Geometry, Algebra and Dynamics'',
funded by the Deutsche Forschungsgemeinschaft.
M.~K.\ is supported by the Berlin Mathematical School.
\end{ack}

\end{document}